\documentclass[twocolumn]{article}

\usepackage[english]{babel}
\usepackage[utf8]{inputenc}

\usepackage{amssymb}
\usepackage{amsmath}
\usepackage{amsthm}
\usepackage{mathtools}
\usepackage{bbm}

\usepackage[a4paper, width=178mm, lines=60]{geometry}

\usepackage{enumerate}
\usepackage{tikz}
\usepackage{marginnote}

\usepackage{authblk}

\newcommand{\bbR}{\mathbb{R}}

\newcommand{\calB}{\mathcal{B}}

\DeclareMathOperator{\id}{id} 
\DeclareMathOperator{\one}{\mathbbm{1}} 
\DeclareMathOperator{\re}{Re} 
\newcommand{\argument}{\mathord{\,\cdot\,}} 
\newcommand{\dx}{\;\mathrm{d}} 
\newcommand{\norm}[1]{\left\lVert #1 \right\rVert} 
\DeclareMathOperator{\dom}{dom} 

\newcommand{\tp}{{\operatorname{T}}}

\newcommand{\spec}{\sigma} 
\newcommand{\spb}{s} 

\newcommand{\rightProof}{``$\Rightarrow$'' }
\newcommand{\leftProof}{``$\Leftarrow$'' }

% ...

\theoremstyle{definition}
\newtheorem{definition}{Definition}
\newtheorem{remark}[definition]{Remark}

\newtheorem{example}[definition]{Example}
\newtheorem{examples}[definition]{Examples}
\newtheorem*{open_problem}{Open Problem}

\theoremstyle{plain}

\newtheorem{theorem}[definition]{Theorem}

\author[$\star$]{Jochen Glück}
\title{Evolution equations with eventually positive solutions}

\affil[$\star$]{Universität Passau, Fakultät für Informatik und Mathematik, 94032 Passau, Germany, jochen.glueck@uni-passau.de}

\begin{document}

\maketitle

\begin{abstract}
	We discuss linear autonomous evolution equations on function spaces which have the property that a positive initial value leads to a solution which initially changes sign, but then becomes -- and stays -- positive again for sufficiently large times. 
	This eventual positivity phenomenon has recently been discovered for various classes of differential equations, but so far a general theory to explain this type of behaviour exists only under additional spectral assumptions.
\end{abstract}

\section{Evolution equations and positivity} 
\label{section:positive}

To set the stage, we start with a reminder about linear evolutions equations whose solutions are positive whenever the initial value is.

\subsection*{Linear ODEs and positivity}

For a matrix $A \in \bbR^{d \times d}$ the linear and autonomous initial value problem
\begin{align*}
	\begin{cases}
		\dot u(t) = Au(t) \quad & \text{for } t \in [0,\infty), \\
		     u(0) = u_0,
	\end{cases}
\end{align*}
where $u_0 \in \bbR^d$, is well-known to be solved by the function $u: [0,\infty) \ni t \mapsto e^{tA}u_0 \in \bbR^d$.
We say that the matrix family $(e^{tA})_{t \in [0,\infty)}$ is \emph{positive} if $e^{tA}u_0 \ge 0$ for all $t \in [0,\infty)$ whenever $u_0 \ge 0$; 
equivalently, $e^{tA} \ge 0$ for all $t \in [0,\infty)$.
Here, we use the notation $\ge 0$ for a vector or a matrix to say that all its entries are $\ge 0$.

\begin{remark}
	There is some terminological inconsistency in the literature with respect to this notion: 
	in matrix analysis and in some parts of PDE theory it is common to use the word \emph{non-negativity}; 
	we use the notion \emph{positivity} instead, which is more common in functional analysis.
\end{remark}

To get an intuition for this concept, it is useful to recall that positivity of the matrix exponential function is easy to characterize in terms of $A$:

\begin{theorem}
	\label{thm:pos-matrix-sg}
	For $A \in \bbR^{d \times d}$ the family $(e^{tA})_{t \in [0,\infty)}$ is positive if and only if every off-diagonal entry of $A$ is $\ge 0$.
\end{theorem}
\begin{proof}
	\rightProof 
	For indices $j \not= k$ one has
	\begin{align*}
		A_{jk} 
		= 
		\lim_{t \downarrow 0} \, \langle e_j, \frac{e^{tA} - \id}{t} e_k \rangle 
		=
		\lim_{t \downarrow 0} \frac{1}{t} \langle e_j, e^{tA} e_k \rangle
		\ge 0,
	\end{align*}
	where $e_j, e_k \in \bbR^d$ are the canonical unit vectors and $\langle \argument, \argument \rangle$ denotes the standard inner product on $\bbR^d$.
	
	\leftProof 
	By assumption one has, for a sufficiently large number $c \ge 0$, the inequality $A + c\id \ge 0$ and hence,
	\begin{align*}
		e^{tA} = e^{-tc} e^{t(A + c\id)} \ge 0
	\end{align*}
	for all $t \in [0,\infty)$, where the inequality at the end follows from the series expansion of the matrix exponential function.
\end{proof}

A typical situation where positivity of matrix exponential functions occurs is the study of Markov processes on finite state spaces:

\begin{example}
	Assume that all off-diagonal entries of $A \in \bbR^{d \times d}$ are $\ge 0$ and that all rows of $A$ sum up to $0$.
	Then $(e^{tA})_{t \in [0,\infty)}$ is positive, and the vector $\one \in \bbR^d$ whose entries are all equal to $1$ satisfies $A \one = 0$ and thus $e^{tA} \one = \one$ for all $t \in [0,\infty)$.
	This shows that each of the matrices $e^{tA}$, $t \ge 0$, is row stochastic, so $(e^{tA})_{t \in [0,\infty)}$ describes a continuous-time Markov process on the finite state space $\{1, \dots, d\}$.
\end{example}

\subsection*{Infinite dimensional equations}

In infinite dimensions, we are still interested in initial value problems of the form
\begin{align*}
	\begin{cases}
		\dot u(t) = Au(t) \quad & \text{for } t \in [0,\infty), \\
		     u(0) = u_0;
	\end{cases}
\end{align*}
but this time, $u_0$ is an element of a Banach space $E$, and $A: E \supseteq \dom(A) \to E$ is a linear operator which is defined on a vector subspace $\dom(A)$ of $E$.
The initial value problem is well-posed if and only if $A$ is a \emph{generator} of a \emph{$C_0$-semigroup} $(e^{tA})_{t \in [0,\infty)}$.
Such a $C_0$-semigroup is a family of bounded linear operators on $E$ which is a suitable infinite dimensional substitute of the matrix exponential function and has similar properties, but it is not given by an exponential series in general.
The solution $u$ to the initial value problem is then given, again, by the formula $u(t) = e^{tA} u_0$ for $t \in [0,\infty)$.
The generator and the $C_0$-semigroup determine each other uniquely, and the relation between semigroup and generator can, in general, be expressed by the formula
\begin{align*}
	\dom(A) & = \{v \in E: \, \lim_{t \downarrow 0} \frac{1}{t}(e^{tA}-\id)v \text{ exists in } E\}, \\ 
	  Av    & = \lim_{t \downarrow 0} \frac{1}{t}(e^{tA}-\id)v.
\end{align*}
The following notion will be used several times later on. 
For a linear operator $A: E \supseteq \dom(A) \to X$ on a Banach space $X$ the quantity
\begin{align*}
	\spb(A) := \sup \big\{ \re \lambda: \; \lambda \in \spec(A) \big\} \in [-\infty, \infty]
\end{align*}
(where $\spec(A)$ denotes the spectrum of $A$) is called the \emph{spectral bound of $A$}.
If $A$ generates a $C_0$-semigroup, then $\spb(A) < \infty$ \cite[Theorem~II.1.10(ii)]{EngelNagel2000}.
More information about $C_0$-semigroup theory can, for instance, be found in the monographs \cite{Pazy1983, EngelNagel2000}.

Let us briefly illustrate the concept of a $C_0$-semigroup by two very classical examples:

\begin{examples}
	\label{exas:c0-sgs}
	(a) 
	Let $p \in (1,\infty)$ and let the operator $A$ be the Laplace operator on the space $L^p(\bbR^n)$, i.e.,
	\begin{align*}
		\dom(A) & = W^{2,p}(\bbR^n), \\
		  Av    & = \Delta v := \sum_{j=1}^n \partial_j^2 v
		\quad \text{for } v \in \dom(A).
	\end{align*}
	Then $A$ generates a $C_0$-semigroup $(e^{tA})_{t \in [0,\infty)}$ on $L^p(\bbR^n)$ that is given by the formula
	\begin{align*}
		(e^{tA}u_0)(x) = \frac{1}{(4\pi t)^{n/2}} \int_{\bbR^n} \exp\left( - \frac{\norm{x-y}_2^2}{4t} \right) u_0(y) \dx y 
	\end{align*}
	for $u_0 \in L^p(\bbR^n)$ and $x \in \bbR^n$.
	The semigroup is called the \emph{heat semigroup} since it describes the solutions to the heat equation
	\begin{align*}
		\dot u(t) = \Delta u(t).
	\end{align*}
	Similar observations can be made on the space $L^1(\bbR^n)$, but the domain of the Laplace operator cannot be chosen to be a Sobolev space then, due to the lack of elliptic regularity.
	
	(b) 
	Let $p \in [1,\infty)$ and let the operator $A$ be the negative first derivative on $L^p(0,\infty)$, given by
	\begin{align*}
		\dom(A) = \left\{v \in W^{1,p}(0,\infty): \; v(0) = 0 \right\}, 
		\quad 
		Av = -v'.
	\end{align*}
	Then $A$ generates the so-called \emph{right shift se\-mi\-group} $(e^{tA})_{t \in [0,\infty)}$ on $L^p(0,\infty)$ given by
	\begin{align*}
		(e^{tA} u_0)(x) = 
		\begin{cases}
			u_0(x-t) \quad & \text{if } t \le x, \\ 
			0        \quad & \text{if } t >   x
		\end{cases}
	\end{align*}
	for $u_0 \in L^p(0,\infty)$. 
	The mapping $u: [0,\infty) \ni t \mapsto e^{tA}u_0 \in L^p(0,\infty)$ is a so-called \emph{mild solution} to the transport equation
	\begin{align*}
		\begin{cases}
			\dot u(t,x) = - \partial_x u(t,x) \quad & \text{for } t,x > 0, \\ 
			     u(0,x) =   u_0(x)          \quad & \text{for } x > 0,   \\ 
			     u(t,0) =   0               \quad & \text{for } t > 0;
		\end{cases}
	\end{align*}
	see \cite[Definition~II.6.3]{EngelNagel2000} for the definition of mild solutions.
	This example is an easy illustration of the general principle that boundary conditions of a PDE are encoded in the domain of the corresponding operator $A$.
\end{examples}

\subsection*{Positive $C_0$-semigroups}

In order to discuss \emph{positive} $C_0$-semigroups, one needs an order structure on the underlying Banach space $E$.
This can, for instance, be a partial order induced by a general closed convex cone, or more specifically the order structure of a Banach lattice.
To facilitate the exposition here, we will restrict our attention to the illustrative case of function spaces though, most importantly to $L^p$-spaces (over $\sigma$-finite measures spaces).

For a function $f \in L^p$ we write $f \ge 0$ to indicate that $f(\omega) \ge 0$ for almost all $\omega$. 
In accordance with the terminology used above we call a function $f$ \emph{positive} if it satisfies $f \ge 0$.
A $C_0$-semigroup $(e^{tA})_{t \in [0,\infty)}$ on $L^p$ is called \emph{positive} if $e^{tA}u_0 \ge 0$ for all $t \in [0,\infty)$ whenever $u_0 \ge 0$. 
Equivalently, each of the operators $e^{tA}$ is positive -- which we denote by $e^{tA} \ge 0$ -- in the sense that it maps positive functions to positive functions.

We have already encountered two examples of positive $C_0$-semigroups: as is easy to see, both semigroups in Examples~\ref{exas:c0-sgs} are positive.

\section{Positivity for large times} 
\label{section:ev-pos-intro}

Let us now proceed to a more surprising situation, where positive initial values lead to solutions which might change sign at first, but again become -- and stay -- positive for sufficiently large times.
In this section we illustrate by means of two easy examples that this kind of behaviour can indeed occur;
a more systematic account is presented in the subsequent section.

\subsection*{A matrix example}

Let us start with a simple three dimensional example.

\begin{example}
	\label{exa:ev-pos-motivating-matrix}
	Consider the orthonormal basis $\calB$ of $\bbR^3$ that consists of the three vectors 
	\begin{align*}
		v_1 = 
		\frac{1}{\sqrt{3}} 
		\begin{pmatrix}
			1 \\ 1 \\ 1
		\end{pmatrix},
		\quad 
		v_2 = 
		\frac{1}{\sqrt{2}}
		\begin{pmatrix}
			-1 \\ 0 \\ 1
		\end{pmatrix},
		\quad 
		v_3 = 
		\frac{1}{\sqrt{6}}
		\begin{pmatrix}
			1 \\ -2 \\ 1
		\end{pmatrix}.
	\end{align*}
	Let $A \in \bbR^{3 \times 3}$ be such that its representation matrix with respect to the basis $\calB$ is given by
	\begin{align*}
		R \coloneqq
		\begin{bmatrix}
			0 &  0 &  0 \\ 
			0 & -1 & -1  \\
			0 &  1 & -1 
		\end{bmatrix},
	\end{align*}
	i.e., we let $A = V R V^{-1}$, where $V \in \bbR^{3 \times 3}$ consists of the columns $v_1, v_2, v_3$. A direct computation shows that $A$ has some strictly negative off-diagonal entries, so $(e^{tA})_{t \in [0,\infty)}$ is not positive according to Theorem~\ref{thm:pos-matrix-sg}. 
	On the other hand, $A$ has the eigenvalue $0$ (with eigenvector $v_1$) as well as the further eigenvalues $-1\pm i$, so $e^{tA}$ converges to the matrix $v_1 \cdot v_1^\tp$, whose entries are all equal to $1/3$, as $t \to \infty$; this shows that $e^{tA}$ is a positive matrix for all sufficiently times $t$.
\end{example}

\subsection*{A fourth order PDE}

Let us now discuss an infinite dimensional example where eventual positivity occurs.

\begin{example}
	\label{exa:fourth-order-pde-1d}
	Let us consider the biharmonic heat equation with periodic boundary conditions on $L^2(0,1)$. 
	It is given by
	\begin{align*}
		\dot u(t) = A u(t) \quad \text{for } t \in [0,\infty),
	\end{align*}
	where $A: L^2(0,1) \supseteq \dom(A) \to L^2(0,1)$ has domain
	\begin{align*}
		& \dom(A) \\
		& = \{v \in H^4(0,1): \, v^{(k)}(0) = v^{(k)}(1) \text{ for } k = 0,1,2,3\}
	\end{align*}
	and is given by $Av = -v^{(4)}$ for each $v \in \dom(A)$.
	The $C_0$-semigroup $(e^{tA})_{t \in [0,\infty)}$ is not positive; this can, for instance, be seen by associating a sesqui-linear form to $-A$ and using the so-called \emph{Beurling--Deny criterion} \cite[Corollary~2.18]{Ouhabaz2005}.
	
	However, we can prove positivity for large times. 
	To this end, note that the operator $A$ is self-adjoint, and its spectrum consists of isolated eigenvalues only since $\dom(A)$ embeds compactly into $L^2(0,1)$. 
	The largest eigenvalue of $A$ is $0$ and the constant function $\one$ spans the corresponding eigenspace. 
	Hence we conclude, for instance from the spectral theorem for self-adjoint operators with compact resolvent, that
	\begin{align*}
		e^{tA}u_0 \to \langle u_0, \one \rangle \one := \int_0^1 u_0(x) \dx x \cdot \one
		\quad \text{in } L^2(0,1)
	\end{align*}
	for each $u_0 \in L^2(0,1)$ as $t \to \infty$.
	Since $A$ is self-adjoint the operators $e^{tA}$ have, for $t > 0$, the property that they map $L^2(0,1)$ into $\dom(A)$ and thus into $L^\infty(0,1)$.
	Moreover they are even continuous from $L^2(0,1)$ to $L^\infty(0,1)$ (this follows for instance from the closed graph theorem), so for $u_0 \in L^2(0,1)$ one even has
	\begin{align*}
		e^{tA}u_0 = e^{1 \cdot A} e^{(t-1)A}u_0 \to \langle u_0, \one \rangle \, e^{1 \cdot A} \one = \langle u_0, \one \rangle \one
	\end{align*}
	as $t \to \infty$, where the convergence takes place with respect to the norm in $L^\infty(0,1)$. 
	This implies that if $u_0 \ge 0$, then $e^{tA}u_0 \ge 0$ for all sufficiently large times $t$.
\end{example}

\section{A systematic theory} 
\label{section:ev-pos-theory}

After the previous ad hoc examples we now present a few excerpts of a more systematic account to eventual positivity.

\subsection*{Eventually positive matrix semigroups}

Example~\ref{exa:ev-pos-motivating-matrix} already gives quite a straightforward idea of how to obtain a sufficient condition for a matrix exponential function to be eventually positive: 
if a matrix $A \in \bbR^{d \times d}$ has a simple real eigenvalue that dominates the real parts of all other eigenvalues and if the corresponding eigenvectors of $A$ and the transposed matrix $A^\tp$ have strictly positive entries only, then we will expect $e^{tA}$ to be positive -- and in fact to even have strictly positive entries only -- for all sufficiently large $t$. 
A bit more surprising is the Perron--Frobenius like fact that the converse implication holds, too.
This was proved by Noutsos and Tsatsomeros in \cite[Theorem~3.3]{NoutsosTsatsomeros2008}, who thus obtained the following theorem (in a slightly different form; see \cite[Theorem~6.1]{DanersGlueckKennedy2016a} for the following version of the result):

\begin{theorem}
	\label{thm:noutsos}
	For a matrix $A \in \bbR^{d \times d}$ the following assertions are equivalent:
	\begin{enumerate}[\upshape (i)]
		\item 
		There exists a time $t_0 \ge 0$ such that all entries of $e^{tA}$ are strictly positive for all $t > t_0$.
		
		\item 
		The spectral bound $\spb(A)$ is a geometrically simple eigenvalue of $A$ and strictly larger than the real part of every other eigenvalue of $A$. 
		Moreover, both $A$ and $A^\tp$ have a strictly positive eigenvector for $\spb(A)$, respectively.
	\end{enumerate}
\end{theorem}

Here, a \emph{strictly positive} vector means a vectors whose entries are all strictly positive.

\subsection*{Individual vs.\ uniform behaviour}

In infinite dimensions, there is a subtlety which we have not properly discussed, yet. 
Let $(e^{tA})_{t \in [0,\infty)}$ be a $C_0$-semigroup on a function space $E$.
If for every $0 \le u_0 \in E$ there exists a time $t_0 \ge 0$ such that $e^{tA}u_0 \ge 0$ for all $t \ge t_0$, 
it is natural to call the semigroup \emph{individually eventually positive} since $t_0$ might depend on $u_0$.
If, in addition, $t_0$ can be chosen to be independent of $u_0$, then we call the semigroup \emph{uniformly eventually positive}. 

In finite dimensions, the two concepts can be easily seen to coincide (just apply the semigroup to all canonical unit vectors), but in infinite dimensions, there exist semigroups which are individually but not uniformly eventually positive; 
see \cite[Examples~5.7 and~5.8]{DanersGlueckKennedy2016a}.

\subsection*{Conditions for eventual positivity in infinite dimensions}

The arguments given in Example~\ref{exa:fourth-order-pde-1d} show individual eventual positivity of the semigroup, and the same argument can easily be generalised to a more abstract setting. 
There is one important issue to note, though: if the leading eigenfunction is not bounded away from $0$, but might be equal to $0$ on the boundary of the underlying domain (as in the case of Dirichlet boundary conditions), then it does no longer suffice for the argument if $e^{1 \cdot A} L^2$ is contained in $L^\infty$ -- instead, one needs that every vector in $e^{1 \cdot A} L^2$ is dominated by a multiple of the leading eigenfunction. 
This property is closely related to Sobolev embedding theorems, and it can be used to give a characterisation of a certain \emph{strong} version of individual eventual positivity that is reminiscent of Theorem~\ref{thm:noutsos}. 

On the other hand, giving conditions for uniform rather than individual eventual positivity is more subtle. 
It requires a domination condition not only on the vectors in the image of $e^{1 \cdot A}L^2$, but also on the image of the dual operator. 
If the semigroup is self-adjoint, though, this dual condition becomes redundant and one ends up with the following sufficient condition for uniform eventual positivity:

\begin{theorem}
	\label{thm:uniform-self-adjoint}
	Let $(\Omega,\mu)$ be a $\sigma$-finite measure space and let $(e^{tA})_{t \in [0,\infty)}$ be a self-adjoint $C_0$-semigroup on $L^2 := L^2(\Omega,\mu)$ which leaves the set of real-valued functions invariant. Let $u \in L^2$ be a function which is strictly positive almost everywhere and assume that the following assumptions hold:
	\begin{enumerate}[\upshape (1)]
		\item 
		The spectral bound $\spb(A)$ is a simple eigenvalue of $A$ and the corresponding eigenspace contains a function $v$ that satisfies $v \ge cu$ for a number $c > 0$.
		
		\item 
		There exists a time $t_1 \ge 0$ such that the modulus of every vector in $e^{t_1 A}L^2$ is dominated by a multiple of $u$.
	\end{enumerate}
	Then $(e^{tA})_{t \in [0,\infty)}$ is uniformly eventually positive.
\end{theorem}

The really interesting part in the conclusion of the theorem is the word \emph{uniformly}, and this is more involved than the argument presented in Example~\ref{exa:fourth-order-pde-1d}. 
Two different proofs of the theorem are known: 
the first one is based on an eigenvalue estimate and the theory of Hilbert--Schmidt operators \cite[Theorem~10.2.1]{Glueck2016} (the assumptions in the reference are slightly different, but the same argument works under the assumptions presented above); the second one employs a duality argument and can thus be generalised to non-self-adjoint semigroups on more general spaces \cite[Theorem~3.3 and Corollary~3.5]{DanersGlueck2018b}. 
This reference also shows that the theorem can be adjusted to even yield a characterisation of a stronger type of eventual positivity.

Theorem~\ref{thm:uniform-self-adjoint} implies the non-trivial observation that the semigroup in Example~\ref{exa:fourth-order-pde-1d} is even uniformly eventually positive.

\subsection*{Spectral properties}

Positive semigroups are known to have surprising structural properties, in particular with regard to their spectrum. 
For various of these properties it can be shown that they are shared by eventually positive semigroups, though some of the proofs are different from the classical proofs for the positive case.
Here are two examples:

\begin{itemize}
	\item 
	If the spectrum of the generator $A$ of an individually eventually positive semigroup is non-empty, then it follows that the spectral bound $\spb(A)$ is itself a spectral value \cite[Theorem~7.6]{DanersGlueckKennedy2016a}.
	
	\item
	For uniformly eventually positive semigroups on $L^p$-spa\-ces, the spectral bound $\spb(A)$ coincides with the so-called \emph{growth bound} of the semigroup (see e.g.\ \cite[Definition~I.5.6]{EngelNagel2000} for a definition); 
	this was recently shown by Vogt \cite[Theorem~2]{VogtPreprint1}.
	The same can be shown, even for individually eventually positive semigroups, on spaces of continuous functions \cite[Theorem~4]{AroraGlueckPreprint1}.
\end{itemize}

More results on the spectrum of eventually positive $C_0$-semigroups can be found in \cite{AroraGlueck2021}.

\section{More examples}

\subsection*{The biharmonic heat equation}

Example~\ref{exa:fourth-order-pde-1d} can be adjusted in the following way: we replace the unit interval with a ball $B$ in $\bbR^d$, the fourth derivative with the square $\Delta^2$ of the Laplace operator and the periodic boundary conditions with so-called \emph{clamped plate} boundary conditions, which require both the function and its normal derivative to vanish at the boundary. 
On $L^2(B)$ this yields the operator $A$ given by
\begin{align*}
	\dom(A) & = H^4(B) \cap H^2_0(B), \\ 
	  Av    & = - \Delta^2 v.
\end{align*}
where $H^4(B)$ and $H^2_0(B)$ denote Sobolev spaces. 
The operator $A$ is self-adjoint and has negative spectral bound. 
It thus generates a $C_0$-semigroup $(e^{tA})_{t \in [0,\infty)}$ which describes the solutions to the so-called \emph{bi-harmonic heat equation}
\begin{align*}
	\dot u(t) = A u(t) 
	\quad \text{for } t \in [0,\infty).
\end{align*}
One has the following result:

\begin{theorem}
	\label{thm:bi-harmonic-clamped-plate}
	The bi-harmonic heat semigroup $(e^{tA})_{t \in [0,\infty)}$ on $L^2(B)$ is uniformly eventually positive.
\end{theorem}

\begin{proof}[Rough outline of the proof]
	Since $B$ is a ball, the inverse operator $(-A)^{-1}$ -- or rather its integral kernel, the so-called \emph{Green function} of $A$ -- can be computed explicitly, and this was in fact done by Boggio over a hundred years ago \cite{Boggio1905} (see also \cite[Section~2]{GrunauSweers1998}).
	The explicit formula shows that $(-A)^{-1}$ maps positive functions to positive functions, and even strengthens their positivity in an appropriate sense. 
	Hence, by a Krein--Rutman type result, one obtains that the leading eigenfunction of $A$ is strictly positive inside $B$.
	Given the specific boundary conditions, it is not too surprising that one also gets that assumptions~(1) and~(2) of Theorem~\ref{thm:uniform-self-adjoint} are satisfied if one chooses $u = d^2$, where $d: B \to [0,\infty)$ describes the distance of each point in $B$ to the boundary $\partial B$.
	Hence, Theorem~\ref{thm:uniform-self-adjoint} gives the desired eventual positivity.
\end{proof}

For more details we refer to \cite[second subsection of Section~6]{DanersGlueckKennedy2016b} and \cite[third subsection of Section~4]{DanersGlueck2018b}. 
A few comments are in order.

\begin{remark}
	\label{rem:biharmonic}
	(a) 
	The argument sketched above breaks down for general domains in $\bbR^d$, since the inverse $(-A)^{-1}$ need no longer be positive in this case. 
	This is a very well-studied topic in PDE theory; see for instance the surveys \cite{Sweers2001} by Sweers and \cite{DallAcquaSweers2004} by Dall'Acqua and Sweers for more information.
	
	(b) 
	However, if one replaces $B$ with a domain which is sufficiently close to a ball, one still obtains the same result.
	The main point here is that positivity of $(-A)^{-1}$ or, under slightly larger perturbations, at least positivity of the leading eigenfunction of $A$ remains true on such domains as shown by Grunau and Sweers in \cite[Theorem~5.2]{GrunauSweers1998}. 
	So Theorem~\ref{thm:bi-harmonic-clamped-plate} holds on this more general class of domains, too.
	
	(c) 
	Theorem~\ref{thm:bi-harmonic-clamped-plate} remains true on general $L^p$-spaces rather than on $L^2$; see for instance \cite[Theorem~4.4]{DanersGlueck2018b}.
	
	(d) 
	If one replaces the clamped plate boundary conditions with so-called \emph{hinged} boundary conditions, which require $u = \Delta u = 0$ on the boundary, the situation becomes much easier because the operator can then be written as minus the square of the Dirichlet Laplace operator. 
	In this case, one has eventual positivity of the semigroup on general domains; on the space of continuous functions, this example is worked out in \cite[Theorem~6.1]{DanersGlueckKennedy2016b}.
\end{remark}

\subsection*{Non-local boundary conditions}

Let us go back to the unit interval and consider the Laplace operator, i.e.\ the second spatial derivative, now.
If we impose local boundary conditions -- such as for instance Dirichlet, Neumann or mixed Dirichlet and Neumann boundary condition, the Laplace operator is well-known to generate a positive semigroup (also on general domains in arbitrary dimension); 
see for instance \cite[Corollary~4.3]{Ouhabaz2005}. 
However, let us now consider an example of non-local boundary conditions instead. 
More specifically, we consider the operator $A$ on $L^2(0,1)$ given by
\begin{align*}
	\dom(A) &= \left\{v \in H^2(0,1): \, v'(0) = -v'(1) = v(0) + v(1) \right\}, \\
	  Av  &= v''.
\end{align*}
This is a self-adjoint operator; the operator, and in particular its relation to the Dirichlet and the Neumann Laplace operator, is discussed in more detail in \cite[Section~3]{Akhlil2018}. 
The spectral bound of $A$ is negative and the inverse $(-A)^{-1}$ can be computed explicitly \cite[proof of Theorem~6.11(i)]{DanersGlueckKennedy2016b}; 
from this formula and the spectral theory of positive operators one can conclude that $\spb(A)$ is a simple eigenvalue and that there is a corresponding eigenfunction which is strictly positive on the closed interval $[0,1]$; see \cite[Theorem~6.11]{DanersGlueckKennedy2016b} for details.
Moreover, one has $e^{1 \cdot A} L^2(0,1) \subseteq \dom(A) \subseteq L^\infty(0,1)$, so the assumptions of Theorem~\ref{thm:uniform-self-adjoint} are satisfied for $u = \one$ and one obtains the following result:

\begin{theorem}
	The semigroup $(e^{tA})_{t \in [0,\infty)}$ on $L^2(0,1)$ generated by the Laplace operator with the nonlocal boundary conditions given above, is uniformly eventually positive.
\end{theorem}

Compare also \cite[Section~4.2]{AroraChillDjidaPreprint1} for a related discussion.
An example of eventual positivity for different nonlocal boundary conditions which lead to a non-self-adjoint realisation of the Laplace operator can be found in \cite[Theorem~4.3]{DanersGlueck2018b}.

\subsection*{Further examples}

Today, eventually positivity, and closely related properties as for instance \emph{asymptotic positivity}, are known for various further $C_0$-semigroups, including the semigroup generated by the Dirich\-let-to-Neumann operator on the unit circle for various parameter choices \cite{Daners2014} (which was the initial motivation for the development of the general theory), several delay differential equations (\cite[Section~6.5]{DanersGlueckKennedy2016a}, \cite[Section~11.6]{Glueck2016} and \cite[Theorem~4.6]{DanersGlueck2018b}), the semigroup generated by a Bi-Laplacian with certain Wentzell boundary conditions \cite[Section~7]{DenkKunzePloss2021}, various semigroups on metric graphs (\cite[Proposition~3.7]{GregorioMugnolo2020a}, \cite[Section~6]{GregorioMugnolo2020b} and \cite[Proposition~5.5]{BeckerGregorioMugnolo2021}) and semigroups generated by Laplacians coupled by point interactions \cite[Proposition~2]{HusseinMugnolo2020}.

\section{Unbounded domains and local properties}
\label{section:local-unbounded}

\subsection*{The biharmonic heat equations on unbounded domains}

A major drawback of Theorem~\ref{thm:uniform-self-adjoint} is that it can only be applied if the leading spectral value is even an eigenvalue of the operator $A$. 
This makes it impossible to apply the theorem to various differential operators that live on unbounded domains.
For instance, consider the biharmonic operator $A$ on $L^2(\bbR^d)$ given by
\begin{align*}
	\dom(A) &= H^4(\bbR^d), \\ 
	  Av  &= -\Delta^2 v.
\end{align*}
The spectrum of $A$ -- which is the set $(-\infty,0]$ -- does not contain eigenvalues, so Theorem~\ref{thm:uniform-self-adjoint} cannot be applied. 
Still, the semigroup $(e^{tA})_{t \in [0,\infty)}$ exhibits a certain local eventual positivity property: for every compact set $K \subseteq \bbR^d$ and every initial value $0 \le u_0 \in L^2(\bbR^d) \cap L^1(\bbR^d)$, the solution $u: t \mapsto e^{tA} u_0$ to the biharmonic heat equation $\dot u(t) = A u(t)$ becomes eventually positive on $K$.
This was proved, under slightly different assumptions on $u_0$, in \cite[Theorem~1(i)]{GazzolaGrunau2008} and \cite[Theorem~1.1(ii)]{FerreroGazzolaGrunau2008} by explicit kernel estimates; 
for more general powers of $\Delta$, a similar result was recently shown in \cite[Theorem~1.1]{FerreiraFerreira2019}. 
Under the assumptions described above, the result was proved by Fourier transform methods in \cite[Theorem~2.1]{DanersGlueckMuiPreprint1}. 

If one replaces the whole space $\bbR^d$ with an infinite cylinder -- for instance of the form $\bbR \times B$, where $B \subseteq \bbR^{d-1}$ is a ball --, and again imposes clamped plate boundary conditions, the same local eventual positivity result remains true. 
The proof is technically more involved, though, and relies on a detailed analysis of the specific partial differential equation under consideration; see \cite[Theorem~2.3 and Section~4]{DanersGlueckMuiPreprint1}.

However, despite the successful analysis of the aforementioned concrete differential equations, an abstract and general theory as outlined in Section~\ref{section:ev-pos-theory} for operators with leading eigenvalue is not yet in sight for the case without eigenvalues:

\begin{open_problem}
	Develop a theory of locally eventually positive $C_0$-semigroups $(e^{tA})_{t \in [0,\infty)}$ which is applicable in situations where the generator $A$ does not have a leading eigenvalue.
\end{open_problem}

\subsection*{Eigenvalues revisited}

Getting back to operators which do have a leading eigenvalue, results such as Theorem~\ref{thm:uniform-self-adjoint} might still not be applicable in some cases due to the conditions~(1) and~(2) which are sometimes particularly subtle at the boundary of $\Omega$ (if $\Omega$ is, say, a domain in $\bbR^d$ and $A$ is a differential operator). 
When all functions are restricted to compact subsets of $\Omega$, though, conditions of the type~(1) and~(2) might still be satisfied. 

This motivates the development of a theory of locally eventually positive semigroups for generators that do have a leading eigenvalue with strictly positive eigenfunction. 
Such a theory was presented by Arora in \cite{AroraToAppear}. 
An application of the theory to certain fourth order operators with unbounded coefficients on $\bbR^d$ (which sometimes have eigenvalues due to the growth of the coefficients) was given in \cite[Section~3.2]{AddonaGregorioRhandiTacelliPreprint1}.

\section{Related topics and results}
\label{section:related}

We close the article by discussing a few related concepts.

\subsection*{Perturbation theory}

If $A$ generates a positive $C_0$-semigroup on a function space $E$, it is quite easy to see that if $B$ is a positive and bounded linear operator on $E$ and $M$ is a bounded and real-valued multiplication operator on $E$, then the perturbed semigroup $(e^{t(A+B+M)})_{t \in [0,\infty)}$ is positive, too:
if $M=0$, this follows for instance from the so-called \emph{Dyson--Phillips series representation} of perturbed semigroups \cite[Theorem~III.1.10]{EngelNagel2000}, and if $M$ is non-zero, it follows from the previous case by using the formula
\begin{align*}
	e^{t(A+B+M)} = e^{-tc} e^{t(A+B+M+c\id)}
\end{align*}
for a real number $c \ge 0$ that is sufficiently large to ensure that $M+c\id$ is positive.

For eventual positivity, though, the situation is much more subtle. 
Under quite general conditions one can show that eventual positivity of a semigroup cannot be preserved by all positive perturbations of the generator.
This was proved in \cite[Theorem~2.3]{DanersGlueck2018a}; related results in finite dimensions had earlier been obtained in \cite[Theorem~3.5 and Proposition~3.6]{ShakeriAlizadeh2017}.
On the other hand, sufficiently small positive perturbations can be shown not to destroy eventual positivity under appropriate assumptions \cite[Section~4]{DanersGlueck2018a}.

\subsection*{Maximum and anti-maximum principles}

One abstract way to formulate that a linear operator $A: E \supseteq \dom(A) \to E$ on a function space $E$ satisfies a \emph{maximum principle} is to require that $(-A)^{-1}$ is a positive operator, i.e.\ maps positive functions to positive functions. 
If $0$ is in the spectrum of $E$, or more generally if the spectral bound of $A$ satisfies $\spb(A) \ge 0$, it is often more natural to consider the \emph{resolvent} $(\lambda \id - A)^{-1}$ for real numbers $\lambda > \spb(A)$.
If the resolvent at one such point $\lambda_0$ is positive, then the same is true for all $\lambda \in (\spb(A), \lambda_0)$, too, and we say that $A$ satisfies a \emph{maximum principle}.
More precisely, this is a \emph{uniform} maximum principle, while we say that $A$ satisfies an \emph{individual maximum principle} if for each $0 \le f \in E$ there exists an ($f$-dependent) number $\lambda_0 > \spb(A)$ such that $(\lambda \id - A)^{-1}f \ge 0$ for all $\lambda \in (\spb(A), \lambda_0)$.

Similarly, it is common to say that $A$ satisfies a \emph{uniform anti-maximum principle} if $\spb(A)$ is, say, an isolated spectral value and for all $\lambda$ in a left neighbourhood of $\spb(A)$ the resolvent $(\lambda \id - A)^{-1}$ maps positive functions to negative functions. 
Likewise one can define an \emph{individual anti-maximum principle} (and clearly, the same concepts can be defined at isolated spectral values different from $\spb(A)$, too). 

Anti-maximum principles have a considerable history and have, for instance, been studied for various elliptic differential operators; see e.g.\ \cite{ClementPeletier1979} for a seminal paper on this topic. 
For biharmonic and polyharmonic operators the validity of (anti-)maximum principles is closely related to the boundary conditions and the geometry of the underlying domain, as explained in Remark~\ref{rem:biharmonic}.

The argument sketched after Theorem~\ref{thm:bi-harmonic-clamped-plate} can be generalised (and partially reversed) in order to obtain a correspondence between the following three types of properties:

\begin{enumerate}[(a)]
	\item 
	Eventual positivity of the semigroup $(e^{tA})_{t \in [0,\infty)}$.
	
	\item 
	Spectral properties of $A$ and positivity of the leading eigenfunction.
	
	\item 
	An individual (anti-)maximum principle for $A$.
\end{enumerate}

This correspondence was discussed in \cite[Sections~3--5]{DanersGlueckKennedy2016b}, where the terminology \emph{eventual positivity and negativity of the resolvent} was used to describe maximum and anti-maxi\-mum principles.
Indeed, equivalence between the three properties (a)--(c) is true under a number of technical restrictions which have been analysed in more detail in \cite{DanersGlueck2017}.

Uniform (anti-)maximum principles are more difficult to analyse than their individual counterparts -- a phenomenon that occurs, as pointed out above, for semigroups, too, but becomes even more pronounced when studying (anti-)maximum principles. 
An abstract operator theoretic approach to uniform anti-maximum principles was first presented by Takáč in \cite[Section~5]{Takac1996}, and recent progress on the topic was made in \cite{AroraGlueck2022}. 
As a sample result, let us discuss the following special case of \cite[Corollary~5.4]{AroraGlueck2022} for self-adjoint operators on $L^2$:

\begin{theorem}
	\label{thm:anti-max}
	Let $(\Omega,\mu)$ be a $\sigma$-finite measure space, let $A: L^2 \supseteq \dom(A) \to L^2$ be a real and self-adjoint operator on $L^2 := L^2(\Omega,\mu)$.
	Let $u \in L^2$ be a function which is $> 0$ almost everywhere and assume that there exists an integer $m \ge 0$ such that every vector in $\dom(A^m)$ is dominated in modulus by a multiple of $u$. 
	Assume moreover that $\lambda_0 \in \bbR$ is an isolated spectral value of $A$ and a simple eigenvalue whose eigenspace contains a function $v$ that satisfies $v \ge cu$ for a number $c > 0$.
	
	If $\mu_1 > \lambda_0$ is in the resolvent set of $A$ and $(\mu_1 \id - A)^{-1} \ge 0$, then the following assertions are equivalent:
	\begin{enumerate}[\upshape (i)]
		\item 
		One has $(\mu \id - A)^{-1} \le 0$ for all $\mu$ in a left neighbourhood of $\lambda_0$.
		
		\item 
		There exists a real number $d > 0$ such that 
		\begin{align*}
			(\id \mu_1 - A)^{-1} f  \le d \, \langle f, u \rangle u
			\quad \text{for all } 0 \le f \in L^2.
		\end{align*}
	\end{enumerate}
\end{theorem}

The assumption that $A$ be a \emph{real operator} means that the domain $\dom(A)$ is spanned by real-valued functions and that $A$ maps real-valued functions to real-valued functions. 
Assertion~(i) of the theorem is a uniform anti-maximum principle, while assertion~(ii) can be considered as an upper kernel estimate for the resolvent (in other words: as an upper Green function estimate) of $A$.
Simple consequences of this theorem are the classical results that the Dirichlet Laplace operator on an interval does not satisfy a uniform anti-maximum principle, while the Neumann Laplace operator on an interval does (see \cite[Proposition~6.1(a) and~(b)]{AroraGlueck2022} for a few more details). 
More involved examples where the theorem (or more general versions thereof) can be applied are discussed in \cite[Section~6]{AroraGlueck2022}.


\bibliographystyle{plain}
\bibliography{literature}

\end{document}